\documentclass[11pt,reqno,a4paper]{amsart}
\usepackage{euscript}

\theoremstyle{plain}
\newtheorem{theorem}{Theorem}
\newtheorem{proposition}{Proposition}
\newtheorem{lemma}{Lemma}
\newtheorem{cor}{Corollary}
\newtheorem{remark}{Remark}
\theoremstyle{remark}

\renewcommand{\epsilon}{\varepsilon}
\renewcommand{\phi}{\varphi}

\def\N{\mathbb{N}}
\def\cA{\EuScript{A}}
\def\R{\mathbb{R}}

\def\Id{\text{\rm Id}}

\begin{document}
\title{Datko-Pazy conditions  for nonuniform exponential stability}

\begin{abstract}
For linear cocycles over both maps and flows, we obtain Datko-Pazy type of conditions under which all Lyapunov exponents of a given cocycle are negative. Furthermore, by
combining our results with the results on subadditive ergodic optimisation, we also present new criteria for uniform exponential stability of linear cocycles. 
\end{abstract}

\begin{thanks}
{The author was supported by an Australian Research Council Discovery Project DP150100017 and in part by the Croatian Science Foundation under the project IP-2014-09-2285. }
\end{thanks}

\author{Davor Dragi\v cevi\'c}
\address{School of Mathematics and Statistics, University of New
South Wales, Sydney NSW 2052, Australia}
\email{d.dragicevic@unsw.edu.au}
\address{Department of Mathematics, University of Rijeka, Croatia}
\email{ddragicevic@math.uniri.hr}

\keywords{Linear cocycle, Lyapunov exponents, Datko-Pazy theorem, Stability}
\subjclass[2010]{Primary: 37D25.}
\maketitle

\section{Introduction}
In the process of extending the Lyapunov operator equation to the case of autonomous systems of the form $x'=Ax$,  where the operator $A$ is unbounded, Datko~\cite{D1}
established his famous theorem which asserts that the trajectories of a $C_0$-semigroup $\{T(t)\}_{t\ge 0}$ on a Hilbert space $X$ exhibit exponential decay if and only if they belong to
$L^2(\R_+, X)$. Since then this result has been recognized as one of the major accomplishments of the modern control theory and has inspired numerous generalizations and extensions.
In particular, Pazy~\cite{P} proved that the conclusion of Datko's theorem is valid even if we replace $L^2(\R_+, X)$ with any $L^p(\R_+, X)$, $p\in [1, \infty)$. Furthermore, Datko~\cite{D2}
obtained related results which deal with the  exponential stability of continuous-time evolution families (see~\cite{Z} for related results in the case of discrete time).
The far-reaching generalization of Datko's theorem for evolution families was obtained by Rolewicz~\cite{R} (see also~\cite{BD, S}). More precisely, Rolewicz 
proved that for an evolution family $T(t, s)_{t\ge s\ge 0}$  of bounded operators on an arbitrary Banach space $X$, the following 
statements are equivalent:
\begin{itemize}
\item $T(t, s)_{t\ge s\ge 0}$ is exponentially stable, i.e. there exist $D, \lambda >0$ such that
\[
\lVert T(t, s)\rVert \le De^{-\lambda (t-s)} \quad \text{for $t\ge s\ge 0$;}
\]
\item  for each $x\in X$, there exists $\alpha (x)>0$ such that
\[
\sup_{s\ge 0} \int_s^\infty N(\alpha (x), \lVert T(t, s)x\rVert )\, dt <\infty.
\]
\end{itemize}
Here, $N \colon (0, \infty) \times [0, \infty) \to [0, \infty)$ is any map  satisfying the following conditions:
\begin{itemize}
\item $N(\cdot, u)$ is non-decreasing for each $u\ge 0$;
\item $N(\alpha, \cdot)$ is continuous and non-decreasing for each $\alpha >0$;
\item $N(\alpha, 0)=0$ and $N(\alpha, u)>0$ for every  $\alpha >0$ and $u\ge 0$. 
\end{itemize}
The above  result has generalized and unified many of the previously known results. 
The more recent contributions  in this line of research deal with certain stochastic extensions (see~\cite{HVV}), weaker forms of (exponential stability) (see~\cite{BLMS, DD, DD1, LP, PPC}) and with 
various relaxations of Datko's condition
that still imply the existence of  exponential stability (see~\cite{PP, RPP}).

In addition, in~\cite{PPP} and~\cite{PPB} the authors have obtained Datko-Pazy type of theorems for linear cocycles. The importance of those results stems
from the fact that the notion of a linear cocycle arises naturally in the study of the nonautonomous dynamical systems. Indeed, the  smooth ergodic theory builds around the study
of the derivative cocycle associated either to a map or a flow (see Chapters 5 and 6 in~\cite{BP}). Moreover, cocycles describe solutions
of variational equations and Cauchy problems with unbounded
coefficients (see Chapter 6 in~\cite{CL}). Finally, we note that cocycles describe solutions of stochastic differential equations (see~\cite{A} for details). Hence, it is of considerable interest to obtain sufficient condition which imply that the cocycle exhibits stable behaviour. 
We refer to~\cite{MSS, Sasu1, Sasu2} for other important works devoted to various characterizations of uniform exponential stability for linear cocycles. 

In the present paper, we formulate Datko-Pazy type of conditions which imply that all Lyapunov exponents of a given linear cocycle are negative. This in turn provides sufficient conditions 
for the existence of a nonuniform exponential stability which includes the classical concept of a (uniform) exponential stability as a very particular case. We emphasize that
nonuniform exponential stability (which corresponds to negativity of all Lyapunov exponents) is a special case of the nonuniform hyperbolicity (which corresponds to the nonvanishing of Lyapunov exponents).
The latter represents an important branch of the modern theory of dynamical systems that originated in the landmark work of Pesin~\cite{P2, P3} (see also~\cite{BP} for detailed exposition).
Our results in this direction are closely related to those in~\cite{PPB} with an important difference that in~\cite{PPB},  sufficient conditions for the existence of a nonuniform exponential stability
are expressed in terms of the so-called Lyapunov norms, while our conditions are formulated in terms of original norm. Consequently, our arguments (unlike those in~\cite{PPB}) rely heavily on ergodic theory.

As a by-product of our results described in previous paragraph, we also extend (for a certain class of cocycles) results in~\cite{PPP} that deal with the  uniform exponential stability of cocycles. More precisely, 
 the assumptions  formulated in~\cite{PPP}
that imply that the cocycle exhibits uniform exponential stability require that certain conditions hold for every point that belongs to the base space $M$ of our dynamics. Here we show that
 assumptions in~\cite{PPP} can be considerably relaxed by assuming that those conditions hold for every point that belongs to the subset $E$ of $M$, which is as large as $M$ from the measure-theoretic
point of view but can be considerably smaller  from let's say dimension theory point of view (see the discussion after Theorem~\ref{de2} for more details). This is achieved by using
useful  results of Morris~\cite{M} on the  so-called subadditive ergodic optimisation. 

The paper is organized as follows. In Section~\ref{COM}, we first give Datko-Pazy type of characterization of cocycles over maps with negative Lyapunov exponents (see Theorem~\ref{de}). 
Afterwards, we show that
negativity of Lyapunov exponents corresponds to the  nonuniform exponential stability even in the Banach-space setting (see Theorem~\ref{de3}). We conclude by providing conditions for the
uniform exponential stability of cocycles that extend those in~\cite{PPP} (see Theorem~\ref{de2}). In Section~\ref{COF}, we obtain the corresponding results for cocycles over semi-flows.

\section{Cocycles over maps}\label{COM}
Let $M$ be an arbitrary compact  metric  space and assume that $f\colon M \to M$ is a continuous map. Furthermore, let $X=(X, \lVert \cdot \rVert)$ be an arbitrary 
separable Banach space and let $B(X)$ denote the space of
all bounded linear operators on $X$. Finally, set 
$\N_0=\{0, 1, 2, \ldots \}$. A map $\cA \colon M \times \N_0 \to B(X)$ is said to be a \emph{cocycle} over $f$ if:
\begin{enumerate}
 \item $\cA(q, 0)=\Id$ for each $q\in M$;
 \item $\cA(q, n+m)=\cA(f^m(q), n)\cA(q, m)$ for each $q\in M$ and $n, m\in \N_0$;
 \item the map $A\colon M \to B(X)$ given by \begin{equation}\label{gen} A(q)=\cA(q, 1), \quad  q\in M \end{equation}is strongly continuous, i.e. the map $q\mapsto A(q)x$ is continuous
 for each $x\in X$. 
\end{enumerate}
We recall that the map $A$ given by~\eqref{gen} is called the \emph{generator} of a cocycle $\cA$. 
Let $\mathcal E(f)$ denote the set of all ergodic, $f$-invariant Borel probability measures on $M$. Since $M$ is compact and $f$ continuous, we have that $\mathcal E(f)\neq \emptyset$.
Observe  that:
\begin{itemize}
\item  the map $q\mapsto \lVert \cA(q, n)\rVert$ is Borel-measurable for each $n\in \mathbb N$ (see~\cite[Lemma 2.4.]{GTQ});
\item  it follows from the strong continuity of $A$, compactness of $M$ and the uniform boundness principle that 
\begin{equation}\label{ww}
\sup_{q\in M} \lVert A(q)\rVert <\infty. 
\end{equation}
\end{itemize}
Hence, the Kingman's subadditive ergodic theorem~\cite{King} implies  that  for each $\mu \in \mathcal E(f)$,
there exists $\lambda_\mu (\cA) \in [-\infty, \infty)$ such that
\begin{equation}\label{jkl}
 \lambda_\mu (\cA)=\lim_{n\to \infty} \frac 1 n \log \lVert \cA(q, n)\rVert, \quad \text{for $\mu$-a.e. $q\in M$.}
\end{equation}
The number $\lambda_\mu (\cA)$ is called the \emph{largest Lyapunov exponent}  of a cocycle $\cA$ with respect to  $\mu$. 
The following result gives a Datko--Pazy type of characterization of cocycles with the property  that $\lambda_\mu(\cA)$ is negative. 
\begin{theorem}\label{de}
For any $\mu \in \mathcal E(f)$, 
 the following properties are equivalent:
 \begin{enumerate}
  \item there exist a Borel-measurable function $C\colon M \to (0, \infty)$ and $p>0$ such that for  $\mu$-a.e. $q\in M$ and $x\in X$, 
  \begin{equation}\label{dpaz}
   \bigg{(} \sum_{n=0}^\infty \lVert \cA(q, n)x\rVert^p \bigg{)}^{1/p} \le C(q) \lVert x\rVert;
  \end{equation}
\item $\lambda_\mu (\cA)<0$.

 \end{enumerate}

\end{theorem}

\begin{proof}
 Assume first that $\lambda_\mu (\cA)<0$ and take an arbitrary $\varepsilon >0$ such that $\lambda_\mu(\cA)+\varepsilon <0$. It follows from~\eqref{jkl} that
 \begin{equation}\label{overc}
  \overline C(q):=\sup  \{  \lVert \cA(q, n)\rVert e^{-n(\lambda_\mu (\cA)+\epsilon)}: n\in \N_0\}<\infty,
 \end{equation}
for $\mu$-a.e. $q\in M$. Obviously, $\overline C$ is measurable and 
\begin{equation}\label{0549q}
 \lVert \cA(q, n)\rVert \le \overline C(q)e^{n(\lambda_\mu (\cA)+\epsilon)} \quad \text{for $\mu$-a.e. $q\in M$ and $n\in \N_0$.}
\end{equation}
Take now any $p>0$ and $x\in X$. It follows from~\eqref{0549q} that 
\[
 \sum_{n=0}^\infty \lVert \cA(q, n)x\rVert^p \le \overline C(q)^p \lVert x\rVert^p \sum_{n=0}^\infty e^{np(\lambda_\mu (\cA)+\varepsilon)}=\frac{\overline C(q)^p}{1-e^{p(\lambda_\mu (\cA)+\varepsilon)}}\lVert x\rVert^p
,\]
and consequently~\eqref{dpaz} holds with 
\[
 C(q):=\frac{\overline C(q)}{(1-e^{p(\lambda_\mu (\cA)+\varepsilon)})^{1/p}}\in (0, \infty).
 \]
 
Conversely, let us  assume that there exist a Borel-measurable function $C\colon M \to (0, \infty)$ and $p>0$ such that~\eqref{dpaz} holds  for  $\mu$-a.e. $q\in M$ and $x\in X$. 
Let
\[
 \lVert x\rVert_{q, p}=\bigg{(} \sum_{n=0}^\infty \lVert \cA(q, n)x\rVert^p \bigg{)}^{1/p} \quad x\in X, \ q\in M.
\]
Noting that $\sum_{n=0}^\infty \lVert \cA(q, n)x\rVert^p \ge \lVert \cA(q, 0)x\rVert^p= \lVert x\rVert^p$, it follows from~\eqref{dpaz} that
\begin{equation}\label{rt}
 \lVert x\rVert \le \lVert x\rVert_{q, p}\le C(q) \lVert x\rVert \quad \text{for $\mu$-a.e. $q\in M$ and every $x\in X$.}
\end{equation}
Since $C$ is measurable,   there exists a Borel set $E\subset M$ such that $\mu(E)>0$ and \begin{equation}\label{K} K:=\sup_{q\in E} C(q)<\infty.\end{equation}
Set
\[
 \gamma:=\bigg{(}1-\frac{1}{K^p}\bigg{)}^{1/p} \in (0, 1).
\]
\begin{lemma}
 For any $m\in \N$, $q\in E$ and $x\in X$, we have that 
 \begin{equation}\label{rt2}
  \lVert \cA(q, m)x\rVert_{f^m (q), p} \le \gamma\lVert x\rVert_{q, p}.
 \end{equation}
\end{lemma}

\begin{proof}[Proof of the lemma]
Note that
\begin{equation}\label{rt3}
 \begin{split}
  \lVert \cA(q, m)x\rVert_{f^m (q), p}^p &=\sum_{n=0}^\infty \lVert \cA(f^m(q), n)\cA(q, m)x\rVert^p \\
  &=\sum_{n=0}^\infty \lVert \cA(q, m+n)x\rVert^p \\
  &=\sum_{n=m}^\infty \lVert \cA(q, n)x\rVert^p \\
  &\le \sum_{n=1}^\infty \lVert \cA(q, n)x\rVert^p \\
  &=\lVert x\rVert_{q, p}^p -\lVert x\rVert^p.
 \end{split}
\end{equation}
On the other hand, \eqref{rt} and~\eqref{K} imply  that 
\[
 \lVert x\rVert_{q, p}^p \le K^p \lVert x\rVert^p,
\]
and thus
\begin{equation}\label{rt1}
 \lVert x\rVert_{q, p}^p -\lVert x\rVert^p \le (1-1/K^p) \lVert x\rVert_{q, p}^p. 
\end{equation}
Combining~\eqref{rt3} and~\eqref{rt1}, we conclude that~\eqref{rt2} holds. 
\end{proof}
On the other hand, it follows from Poincare recurrence theorem (see~\cite{Walters}) that  $\mu(E')=\mu(E)$,
where
\[
 E':=\{ q\in E: \text{$f^n (q)\in E$ for infinitely many $n\in \N$}\}.
\]
For each $q\in E'$, set 
\[
 \tau(q):=\min \{n\in \N: f^n(q)\in E\} \quad \text{and} \quad \bar f(q):=f^{\tau (q)} (q).
\]
Moreover, let  \[ \overline A(q):=\cA(q, \tau(q)), \quad  q\in E'\] and consider the cocycle $\overline{\cA}$ over $\bar f$ and with generator $\overline A$. Note that
\begin{equation}\label{coc}
 \overline{f}^n (q)=f^{\tau_n(q)}(q) \quad \text{and} \quad \overline{\cA}(q, n)=\cA(q, \tau_n(q)) \quad \text{for $q\in E'$ and $n\in \mathbb N$,}
\end{equation}
where \[\tau_n(q):=\sum_{i=0}^{n-1} \tau (\bar f^i (q)).\]
\begin{lemma}
 For each $q\in E'$ and $n\in \N$, we have that
 \begin{equation}\label{gh}
  \lVert \overline{\cA}(q, n)x\rVert_{\overline f^n(q), p} \le \gamma^n \lVert x\rVert_{q, p} \quad \text{for every $x\in X$.}
 \end{equation}
\end{lemma}
\begin{proof}[Proof of the lemma]
By~\eqref{rt2} and~\eqref{coc}, we have that
\[
\begin{split}
  \lVert \overline{\cA}(q, n)x\rVert_{\overline f^n(q), p} &=\lVert \cA(q, \tau_n(q))x\rVert_{f^{\tau_n(q)}(q), p}\\
  &=\lVert \cA(f^{\tau_{n-1}(q)}(q), \tau_n(q)-\tau_{n-1}(q))\cA(q, \tau_{n-1}(q))x\rVert_{f^{\tau_n(q)}(q), p}\\
  &\le \gamma \lVert \cA(q, \tau_{n-1}(q))x\rVert_{f^{\tau_{n-1}(q)}(q), p}.
\end{split}
  \]
By iterating, we conclude that~\eqref{gh} holds. 
\end{proof}
It follows from~\eqref{rt}, \eqref{K} and~\eqref{gh} that 
\[
 \lVert \overline{\cA}(q, n)x\rVert \le \lVert \overline{\cA}(q, n)x\rVert_{\overline f^n(q), p} \le \gamma^n \lVert x\rVert_{q, p} \le K\gamma^n \lVert x\rVert,
\]
for every $x\in X$, $n\in \N$ and $q\in E'$. Therefore,
\begin{equation}\label{on}
 \lVert \overline{\cA}(q, n)\rVert \le K\gamma^n \quad \text{for $q\in E'$ and $n\in \N$.}
\end{equation}
By~\eqref{on}, 
\begin{equation}\label{zy}
 \limsup_{n\to \infty} \frac 1 n \log \lVert \overline{\cA}(q, n)\rVert <0 \quad \text{for $q\in E'$.}
\end{equation}
Consequently, for $\mu$-a.e. $q\in E'$ we have (using~\eqref{coc} and~\eqref{zy})
\begin{equation}\label{ok}
 \begin{split}
  \limsup_{n\to \infty} \frac{1}{\tau_n(q)} \log \lVert \cA(q, \tau_n(q))\rVert 
  &=\limsup_{n\to \infty} \frac{1}{\tau_n(q)} \log \lVert \overline{\cA}(q, n)\rVert \\
  &=\lim_{n\to \infty} \frac{n}{\tau_n(q)} \cdot \limsup_{n\to \infty} \frac{1}{n} \log \lVert \overline{\cA}(q, n)\rVert \\
  &=\mu(E)\cdot  \limsup_{n\to \infty} \frac{1}{n} \log \lVert \overline{\cA}(q, n)\rVert \\
  &<0,
 \end{split}
\end{equation}
where we have also used Kac's lemma (see~\cite{Walters}) which implies  that 
\[
 \lim_{n\to \infty} \frac{\tau_n(q)}{n}=\frac{1}{\mu(E)} \quad \text{for $\mu$-a.e. $q\in E'$.}
\]
Finally, since for $\mu$-a.e. $q\in E'$ we have 
\[
 \lambda_{\mu}(\cA)=\lim_{n\to \infty} \frac 1 n \log \lVert \cA(q, n)\rVert =\lim_{n\to \infty} \frac{1}{\tau_n(q)} \log \lVert \cA(q, \tau_n(q))\rVert,
\]
it follows immediately from~\eqref{ok} that 
 $\lambda_\mu (\cA)<0$. 
\end{proof}
The importance of Theorem~\ref{de} stems from the fact that the negativity of the largest Lyapunov exponent corresponds to the concept of the  \emph{nonuniform} exponential stability.
More precisely, we have the following result.
\begin{theorem}\label{de3}
Assume that  that $\lambda_\mu(\cA)<0$ for some $\mu \in \mathcal E(f)$. Then, for each $\epsilon >0$ there exists a measurable function $T\colon M \to (0, \infty)$ such that:
\begin{enumerate}
 \item for $\mu$-a.e. $q\in M$ and $n\in \N_0$, 
 \begin{equation}\label{1}
  \lVert \cA(q, n)\rVert \le T(q)e^{(\lambda_\mu(\cA)+\epsilon)n};
 \end{equation}
\item for $\mu$-a.e. $q\in M$ and $n\in \N_0$, 
\begin{equation}\label{2}
 T(f^n(q)) \le T(q)e^{\epsilon n}.
\end{equation}

\end{enumerate}

\end{theorem}

\begin{proof}
 We follow closely arguments in~\cite[Theorem 2.]{DF} (which are in turn inspired by classical ergodic theory arguments). Let $\overline C$ be as in~\eqref{overc}. 
 \begin{lemma}
  We have that
  \begin{equation}\label{temp}
   \lim_{n\to \infty} \frac 1 n \log \overline{C} (f^n(q))=0, \quad \text{for $\mu$-a.e. $q\in M$.}
  \end{equation}

 \end{lemma}
 
 \begin{proof}[Proof of the lemma]
 Since
 \[
  \lVert \cA(q, n)\rVert \le \lVert \cA(f(q), n-1)\rVert \cdot \lVert A(q)\rVert,
 \]
we have that
\[
 e^{-n(\lambda_\mu (\cA)+\epsilon)} \lVert \cA(q, n)\rVert \le  e^{-(n-1)(\lambda_\mu (\cA)+\epsilon)}\lVert \cA(f(q), n-1)\rVert \cdot e^{-(\lambda_\mu(\cA)+\epsilon)}\lVert A(q)\rVert.
\]
Hence,
\[
 \overline{C}(q)\le \overline{C}(f(q))\cdot \max \{e^{-(\lambda_\mu(\cA)+\epsilon)}\lVert A(q)\rVert, 1\}.
\]
It follows from~\eqref{ww} that there exists $\psi>0$ such that
\[
 \log \overline{C}(q)-\log \overline{C}(f(q))\le \psi. 
\]
Set 
\[
 D(q):=\log \overline{C}(q)-\log \overline{C}(f(q)).
\]
We note that 
\begin{equation}\label{sk}
 \frac 1 n \log \overline{C}(f^n(q))=\frac 1 n \log \overline{C}(q)-\frac  1 n \sum_{j=0}^{n-1} D(f^j(q)),
\end{equation}
for $q\in M$ and $n\in \N$. Note that $D^+:=\max \{0, D\}$ is integrable. It follows from  Birkhoff's ergodic theorem, there exists $a\in [-\infty, \infty)$ such that
\begin{equation}\label{sk1}
 \lim_{n\to \infty} \frac 1 n \sum_{j=0}^{n-1} D(f^j(q))=a
\end{equation}
for $\mu$-a.e. $q\in M$. It follows from~\eqref{sk} and~\eqref{sk1} that
\[
 \lim_{n\to \infty}\frac{1}{n} \log \overline{C}(f^n(q))=-a
\]
On the other hand, since $\mu$ is $f$-invariant, for any $c>0$ we have
\[
\begin{split}
 \lim_{n\to \infty} \mu(\{q\in M: \log \overline{C}(f^n(q))/n \ge c\}) &=\lim_{n\to \infty} \mu(\{q\in M: \log \overline{C}(q)\ge nc\})\\
 &=0,
 \end{split}
\]
which implies that $a\ge 0$. Therefore, 
\[
 \lim_{n\to \infty}\frac 1 n  \log \overline{C}(f^n(q))\le 0 \quad \text{for $\mu$-a.e. $q\in M$.}
\]
Since $\overline{C}(q)\ge 1$, we conclude that~\eqref{temp} holds. 
 \end{proof}
It follows from~\cite[Proposition 4.3.3.]{A} and~\eqref{temp} that  there exists a measurable function $T \colon M \to (0, \infty)$ satisfying~\eqref{2}  and such that $\overline{C}(q)\le T(q)$ for
$\mu$-a.e. $q\in M$. Together with~\eqref{0549q} this implies that~\eqref{1} also holds. 
\end{proof}
\begin{remark}
We emphasize that the concept of nonuniform exponential stability is a particular case of the notion of a tempered exponential dichotomy which was recently studied in a  Banach space  setting in~\cite{ZLZ} and~\cite{BDV}.
\end{remark}

We shall now show that Theorem~\ref{de} when combined with the result obtained by Morris~\cite{M} (building on the earlier work of   Schreiber~\cite{SJS} and Sturman and Stark~\cite{SS}) gives also new conditions for \emph{uniform} exponential stability of continuous cocycles, i.e. of cocycles with the property that $A\colon M \to B(X)$ is a continuous map. 
We will say that a Borel subset $E\subset M$ has full-measure  if $\mu(E)=1$ for every $\mu \in \mathcal {E} (f)$.
\begin{theorem}\label{de2}
 Assume that $\cA$ is a cocycle over $f$ such that the map $A$ given by~\eqref{gen} is continuous.  Then, the following properties are equivalent:
 \begin{enumerate}
  \item there exist a Borel-measurable function $C\colon M \to (0, \infty)$, a full-measure set $E\subset M$ and $p>0$ such that~\eqref{dpaz} holds for each $q\in E$ and $x\in X$;
 \item $\cA$ is uniformly exponentially stable, i.e. there exist $D, \lambda >0$ such that
 \begin{equation}\label{us}
  \lVert \cA(q, n)\rVert \le De^{-\lambda n} \quad \text{for every $q\in M$ and $n\in \N_0$.}
 \end{equation}
\end{enumerate}

\end{theorem}

\begin{proof}
 Proceeding as in the proof of Theorem~\ref{de}, it is easy to verify that~\eqref{us} implies that~\eqref{dpaz} holds for any $p>0$, $q\in 
 M$, $x\in X$ and with a  constant function $C \colon M \to (0, \infty)$. 
 
 Let us establish the converse. Assume that there exist a Borel-measurable function $C\colon M \to (0, \infty)$, a full-measure set $E\subset M$ and $p>0$ 
 such that~\eqref{dpaz} 
 holds for each  $q\in E$ and $x\in X$. It follows from Theorem~\ref{de} that 
 \begin{equation}\label{neg}
  \lambda_{\mu}(\cA)<0 \quad \text{for every $\mu \in \mathcal E(f)$.}
 \end{equation}
Consider a sequence of maps $(F_n)_{n\in \N}$, where  $F_n \colon M \to \R \cup \{-\infty\}$ is  given  by
\begin{equation}\label{Fn}
 F_n(q):=\log \lVert \cA(q, n)\rVert \quad \text{for $q\in M$ and $n\in \N$,}
\end{equation}
with the convention that $\log 0:=-\infty$. 
Note that since $A$ and $f$ are continuous, $F_n$ is upper semi-continuous for each $n\in \N$. Moreover,
\[
 F_{n+m}(q)\le F_n(f^m(q))+F_m(q), \quad \text{for $m, n\in \N$ and $q\in M$.}
\]
Therefore, it follows from~\cite[Theorem A.3.]{M} that there exists $\nu \in \mathcal E(f)$ such that
\[
 \lim_{n\to \infty} \frac 1 n  \max_{q\in M} F_n(q) =\lambda_{\nu}(\cA),
\]
which together with~\eqref{neg} implies that
\[
  \lim_{n\to \infty} \frac 1 n  \max_{q\in M} F_n(q)<0. 
\]
Hence, there exist $\lambda >0$ and $n_0\in \N$ such that
\[
 \max_{q\in M} F_n(q) \le -\lambda n \quad \text{for every $n\ge n_0$},
\]
which in a view of~\eqref{Fn} readily implies~\eqref{us}.
\end{proof}

Finally, we note that Theorem~\ref{de2} generalizes (for continuous cocyles) the following result established in~\cite{PPP}.
\begin{cor}\label{dpt}
  Assume that $\cA$ is a cocycle over $f$ such that the map $A$ given by~\eqref{gen} is continuous.  Then, the following properties are equivalent:
 \begin{enumerate}
  \item there exist $C>0$ and  $p>0$ such that for $q\in M$ and $x\in X$,
  \begin{equation}\label{02}
      \bigg{(} \sum_{n=0}^\infty \lVert \cA(q, n)x\rVert^p \bigg{)}^{1/p} \le C \lVert x\rVert;
\end{equation}
\item  there exist $D, \lambda >0$ such that~\eqref{us} holds.
\end{enumerate}
\end{cor}

\begin{proof}
 The proof follows by applying Theorem~\ref{de2} for $E=M$ (which is obviously set of full-measure) and $C(q)=C$, $q\in M$.
\end{proof}

\begin{remark}
We would now like to  compare Theorem~\ref{de2} and Corollary~\ref{dpt}. More precisely, we want to explain why Theorem~\ref{de2} represents a nontrivial and meaningful
extension of Corollary~\ref{dpt}. We first note that in the statement of Corollary~\ref{dpt} it is required that~\eqref{02} holds for \emph{every} $q\in M$. On the other hand, in Theorem~\ref{de2}  
we required that~\eqref{dpaz} holds for each $q\in E$, where $E$ is a set of full-measure. This means that $M\setminus E$ is negligible from measure-theoretic point of view, i.e. $\mu(M\setminus E)=0$ for
each $\mu \in \mathcal E(f)$. However, it follows from results in~\cite{BS}  that in many generic situations, $M\setminus E$ can carry full topological entropy and full Hausdorff dimension, i.e.
can be as large as the whole space $M$ with respect to those two notions. In particular,  in such situations  $M\setminus E$ is non-empty (and in fact uncountable). 
Finally, we note that in the statement of Theorem~\ref{de2}, we have required the existence of a measurable, not necessarily constant  function $C\colon M \to (0, \infty)$ such that~\eqref{dpaz} holds. 
On the other hand, the condition~\eqref{02} involves the same $C>0$ for all $q\in M$. We conclude that Theorem~\ref{de2} represents  a substantial and nontrivial  generalization of Corollary~\ref{dpt}.
\end{remark}

\begin{remark}
We would also like to emphasize that Theorem~\ref{de2} deals with continuous cocycles, while all results in~\cite{PPP} consider strongly continuous cocycles.  Certainly, the assumption that the cocycle is continuous is more restrictive. However, it still  includes many interesting classes of dynamics that arise in applications (see~\cite[Chapter 8.]{CL}).
\end{remark}
\section{Cocycles over semi-flows}\label{COF}

We continue to assume that $M$ is a compact metric  space and that $X$ is a separable  Banach space.  A family of maps $\Phi=(\phi_t)_{t\ge 0}$, $\phi_t \colon M \to M$ is said to be a \emph{semi-flow} if:
\begin{enumerate}
 \item $\phi_t$ is a continuous map for each $t\ge 0$;
 \item $\phi_0=\Id$;
 \item $\phi_{t+s}=\phi_t \circ \phi_s$ for $t, s\ge 0$;
\item the map $(q, t)\mapsto \phi^t(q)$ is continuous on $ M\times [0, \infty)$. 
\end{enumerate}
Furthermore, we say that the map $\cA \colon M \times [0, \infty) \to B(X)$ is a \emph{cocycle} over $\Phi=(\phi_t)_{t\ge 0}$ if:
\begin{enumerate}
 \item $\cA(q, 0)=\Id$ for $q\in M$;
 \item $\cA(q, t+s)=\cA(\phi_s(q), t)\cA(q, s)$ for $q\in M$ and $t, s\ge 0$;
 \item $(q, t)\mapsto \cA(q, t)x$ is a continuous map on $M\times [0, \infty)$ for each $x\in X$.
\end{enumerate}
It follows easily from the uniform boundness principle and the assumption that $M$ is compact that $\cA$ is exponentially bounded, i.e. that there  $K, \omega >0$ such that
  \begin{equation}\label{ubg}
   \lVert \cA(q, t)\rVert \le Ke^{\omega t} \quad \text{for  $q\in M$ and $t\ge 0$.}
  \end{equation}
Let $\mathcal E(\Phi)$ denote the space of all ergodic, $\Phi$-invariant  Borel probability measures. 
It follows from Kingman's subadditive ergodic theorem~\cite{King} that  for each $\mu \in \mathcal E(\Phi)$,
there exists $\lambda_\mu (\cA) \in [-\infty, \infty)$ such that
\begin{equation}\label{jklc}
 \lambda_\mu (\cA)=\lim_{t\to \infty} \frac 1 t \log \lVert \cA(q, t)\rVert, \quad \text{for $\mu$-a.e. $q\in M$.}
\end{equation}
As for cocycles over maps, the number $\lambda_\mu (\cA)$ is called the \emph{largest Lyapunov exponent}  of the cocycle $\cA$ with respect to  $\mu$. 
\begin{proposition}\label{propd}
 Assume that $\lambda_\mu(\cA)<0$. Then for any $p>0$, there exists a measurable function $C\colon M \to (0, \infty)$ such that 
 \begin{equation}\label{po}
  \bigg{(} \int_0^\infty \lVert \cA(q, t)x\rVert^p \, dt \bigg{)}^{1/p} \le C(q)\lVert x\rVert 
 \end{equation}
for each $x\in X$ and $\mu$-a.e. $q\in M$.
\end{proposition}

\begin{proof}
  Assume first that $\lambda_\mu (\cA)<0$ and take an arbitrary $\varepsilon >0$ such that $\lambda_\mu(\cA)+\varepsilon <0$. It follows from~\eqref{jklc}  that 
 \begin{equation}\label{def}
  \overline C(q):=\sup  \{  \lVert \cA(q, t)\rVert e^{-t(\lambda_\mu (\cA)+\epsilon)}: t\ge 0\}<\infty,
 \end{equation}
for $\mu$-a.e. $q\in M$.
Obviously,  
\begin{equation}\label{0549c}
 \lVert \cA(q, t)\rVert \le \overline C(q)e^{t(\lambda_\mu (\cA)+\epsilon)} \quad \text{for $\mu$-a.e. $q\in M$ and $t\ge 0$.}
\end{equation}
For any $p>0$, using~\eqref{0549c} we have that
\[
 \begin{split}
  \int_0^\infty \lVert \cA(q, t)x\rVert^p \, dt \le \overline{C}(q)^p \lVert x\rVert^p \int_0^\infty e^{pt(\lambda_\mu (\cA)+\epsilon)}\, dt=
  \frac{\overline{C}(q)^p}{-p(\lambda_\mu(\cA)+\epsilon)}\lVert x\rVert^p
 \end{split}
\]
for $\mu$-a.e. $q\in M$ and $x\in X$. Hence, \eqref{po} holds with
\[
 C(q):=\frac{\overline{C}(q)}{(-p(\lambda_\mu(\cA)+\epsilon))^{1/p}}, \quad q\in M. 
\]

\end{proof}

We now establish the converse to Proposition~\ref{propd}. 

\begin{theorem}\label{cthm}
 Assume that there exist $p>0$ and a measurable function $C \colon M \to (0, \infty)$ such that~\eqref{po} holds for each $x\in X$ and $\mu$-a.e. $q\in M$.
  Then, $\lambda_\mu(\cA)<0$. 

\end{theorem}

\begin{proof}
 If follows from~\eqref{ubg} that
 \[
  \lVert \cA(q, n+1)x\rVert \le \lVert \cA(\phi_t(q), n+1-t)\rVert \cdot \lVert \cA(q, t)x\rVert \le Ke^\omega \lVert \cA(q, t)x\rVert,
 \]
for $\mu$-a.e. $q\in M$, $n\in \N_0$, $t\in [n, n+1]$ and $x\in X$. Thus,
\[
 \lVert \cA(q, n+1)x\rVert^p \le K^p e^{\omega p} \int_n^{n+1} \lVert \cA(q, t)x\rVert^p \, dt
\]
and consequently
\[
 \sum_{n=1}^\infty \lVert \cA(q, n)x\rVert^p \le K^p e^{\omega p} \int_0^\infty \lVert \cA(q, t)x\rVert^p \, dt,
\]
for $\mu$-a.e. $q\in M$ and every $x\in X$. It follows from~\eqref{po} that 
\begin{equation}\label{gf}
 \bigg{(}  \sum_{n=0}^\infty \lVert \cA(q, n)x\rVert^p \bigg{)}^{1/p} \le (K^p e^{\omega p}C(q)^p+1)^{1/p} \lVert x\rVert, 
\end{equation}
for $\mu$-a.e. $q\in M$ and every $x\in X$. Note that the restriction of $\cA$ to $M \times \N_0$ is a cocycle over $\phi_1$ and that $\mu \in \mathcal E(\phi_1)$. 
Hence, it follows from Theorem~\ref{de} and~\eqref{gf} that 
\[
 \lim_{n\to \infty} \frac 1 n \log \lVert \cA(q, n)\rVert <0 \quad \text{for $\mu$-a.e. $q\in M$,}
\]
which implies that \[
  \lambda_{\mu}(\cA) =\lim_{t\to \infty} \frac 1 t \log \lVert \cA(q, t)\rVert=\lim_{n\to \infty} \frac 1 n \log \lVert \cA(q, n)\rVert<0.
                  \]
\end{proof}
The following is a continuous-time version of Theorem~\ref{de2}. 
\begin{theorem}\label{esr}
 Assume that $\cA$ is a continuous cocycle over $\Phi$. Furthermore, suppose that $M$ is a compact topological space. Then, the following properties are equivalent:
 \begin{enumerate}
  \item there exist a Borel-measurable function $C\colon M \to (0, \infty)$, a full-measure set $E\subset M$ and $p>0$ such that~\eqref{po} holds for each $q\in E$ and $x\in X$;
 \item $\cA$ is uniformly exponentially stable, i.e. there exist $D, \lambda >0$ such that
 \begin{equation}\label{usc}
  \lVert \cA(q, t)\rVert \le De^{-\lambda t} \quad \text{for every $q\in M$ and $t\ge 0$.}
 \end{equation}
 \end{enumerate}
\end{theorem}

\begin{proof}
 We shall show that (1) implies (2) since the converse is easy to show.  
 It follows from Theorem~\ref{cthm} that
 \begin{equation}\label{sw}
  \lambda_\mu(\cA)<0 \quad \text{for every $\mu \in \mathcal E(\Phi)$.}
 \end{equation}
For each $t\ge 0$, we define $F_t \colon M \to \mathbb R\cup \{-\infty\}$ by
\[
 F_t(q)=\log \lVert \cA(q, t)\rVert, \quad q\in M. 
\]
Note that $F_t$ is upper semi-continuous and that
\[
 F_{t+s}(q)\le F_t(\phi_s(q))+F_s(q), \quad \text{for $q\in M$ and $t, s\ge 0$.}
\]
It follows from~\eqref{sw} and~\cite[Theorem A.3.]{M} that
\[
 \lim_{t\to \infty} \frac 1 t \log \max_{q\in M} \lVert \cA(q, t)\rVert<0,
\]
which immediately implies~\eqref{usc}.
\end{proof}
The following result is a  direct consequence of the previous theorem. For the general case of strongly continuous cocycles it has been proved in~\cite[Theorem 2.3.]{PPP}.
\begin{cor}
  Assume that $\cA$ is a cocycle over $\Phi$ which is continuous.  Then, the following properties are equivalent:
  \begin{enumerate}
   \item there exist $C, p>0$ such that
   \[
    \bigg{(} \int_0^\infty \lVert \cA(q, t)x\rVert^p \, dt \bigg{)}^{1/p} \le C \lVert x\rVert,
   \]
   for every $q\in M$ and $x\in X$;
\item there exist $D, \lambda >0$ such that~\eqref{usc} holds. 
  \end{enumerate}

\end{cor}


\begin{thebibliography}{11}
\bibitem{A} L. Arnold, \emph{Random dynamical systems}, Springer Monographs in Mathematics, Springer\--Verlag, Berlin, 1998. 
\bibitem{BDV} L. Barreira, D. Dragi\v cevi\' c and C. Valls, \emph{Tempered exponential dichotomies: admissibility and stability under perturbations}, Dyn. Syst. \textbf{31} (2016), 525--545.
\bibitem{BP} L. Barreira and Ya. Pesin, \emph{Nonuniform Hyperbolicity}, Encyclopedia of Mathematics and its Applications 115, Cambridge University Press, 2007.
\bibitem{BS} L.~Barreira and J.~Schmeling, \emph{Sets of ``non-typical'' points have full topological entropy and full Hausdorff dimension}, Israel J. Math. \textbf{116} (2000), 29--70.
\bibitem{BLMS} A. Bento, N. Lupa, M. Megan and C. Silva, \emph{Integral conditions for nonuniform $\mu$-dichotomy on the half--line}, Discrete Contin. Dyn. Syst. Ser. B \textbf{22} (2017), 3063--3077.
\bibitem{BD} C. Buse and S. Dragomir, \emph{New characterizations o
f asymptotic stability for evolution families on
Banach spaces}, Electronic J. Differential Equations \textbf{38} (2004), 9pp.
\bibitem{CL} C. Chicone and Yu. Latushkin, \emph{Evolution semigroups in dynamical systems and differential equations}, Mathematical Surveys and Monographs 70, American Mathematical Society, Providence, RI, 1999.
\bibitem{D1} R. Datko, \emph{Extending a theorem of A. M. Liapunov to Hilbert space}, J. Math. Anal. Appl. \textbf{32} (1970), 610-616. 
\bibitem{D2} R. Datko, \emph{Uniform asymptotic stability of evolutionary processes in Banach space}, SIAM J. Math. Anal. \textbf{3} (1972), 428--445.
\bibitem{DD} D. Dragi\v cevi\' c, \emph{A version of a theorem of R. Datko for stability in average}, Systems Control Lett. \textbf{96} (2016), 1--6.
\bibitem{DD1}  D. Dragi\v cevi\' c, \emph{Strong nonuniform behaviour: A Datko type characterization}, J. Math. Anal. Appl., in press. 
\bibitem{DF} D. Dragi\v cevi\' c and G. Froyland, \emph{H\"older continuity of Oseledets splittings for semi-invertible operator cocycles},  Ergodic Theory Dynam. Systems, in press.
\bibitem{GTQ} C.  Gonz\'alez-Tokman and A. Quas, \emph{A semi-invertible operator Oseledets theorem}, Ergodic Theory and Dynamical Systems \textbf{34} (2014), 1230--1272.
\bibitem{HVV} B. Haak, J. van Neerven and M. Veraar, \emph{A stochastic Datko-Pazy theorem}, J. Math. Anal. Appl. \textbf{329} (2007), 1230--1239.
\bibitem{LP} N. Lupa and L. H. Popescu, \emph{Banach function spaces and Datko-type conditions for nonuniform exponential stability of evolution families}, preprint. 
\bibitem{King}J.F.C. Kingman, \emph{Sub-additive ergodic theory}, Ann. Probab. \textbf{1} (1973), 883--909.
\bibitem{MSS} M. Megan, A. L. Sasu and B. Sasu, \emph{On uniform exponential stability of linear
skew-product semiflows in Banach spaces}, Bull. Belg. Math. Soc. Simon Stevin \textbf{9} (2002), 143--154.
\bibitem{M} I. D.  Morris, \emph{Mather sets for sequences of matrices and applications to the study of joint spectral radii}, Proc. Lond.  Math. Soc. \textbf{107} (2013), 121--150. 
\bibitem{P} A. Pazy, \emph{On the applicability of Lyapunov’s theorem in Hilbert space}, SIAM J. Math. Anal. \textbf{3} (1972), 291--294.
\bibitem{P2} Ya.~Pesin, \emph{Families of invariant manifolds corresponding to nonzero characteristic exponents}, Math. USSR-Izv. \textbf{40} (1976), 1261--1305.
\bibitem{P3} Ya.~Pesin, \emph{Characteristic Ljapunov exponents, and smooth ergodic theory}, Russian Math. Surveys \textbf{32} (1977), 55--114.
\bibitem{PP} C. Preda and P. Preda, \emph{An ergodic theorem in the qualitative behaviour of (non)linear semiflows}, Nonlinear Anal. \textbf{75} (2012), 5393--5400. 
\bibitem{PPB} C. Preda, P. Preda and F. B\v{a}t\v{a}ran, \emph{An extension of a theorem of R. Datko to the case of (non)uniform exponential stability of linear skew-product semiflows},
J. Math. Anal. Appl. \textbf{425} (2015), 1148--1154.
\bibitem{PPC} C. Preda, P. Preda and A. Craciunescu, \emph{A version of a theorem of R. Datko for nonuniform exponential contractions}, J. Math. Anal. Appl. \textbf{385} (2012), 572--581.
\bibitem{PPP} C. Preda, P. Preda and A. Petre, \emph{On the asymptotic behavior of an exponentially bounded, strongly continuous cocycle over a semiflow}, Comm. Pure Appl. Anal. \textbf{8}
(2009), 1637--1645.
\bibitem{RPP} R. O. Mo\c{s}incat, C. Preda and P. Preda, \emph{Averaging theorems for the large-time behavior of the solutions of nonautonomous systems}Systems Control Lett. \textbf{60} (2011), 994-999.
\bibitem{R} S. Rolewicz, \emph{On uniform N-equistability}, J. Math. Anal. Appl. \textbf{115} (1986), 434--441.
\bibitem{Sasu1} A. L. Sasu and B. Sasu, \emph{Exponential stability for linear skew-product flows}, Bull. Sci. Math. \textbf{128} (2004), 727--738.
\bibitem{Sasu2} B. Sasu, \emph{On exponential dichotomy of variational difference equations}, Discrete Dyn. Nat. Soc. (2009), 324273, 18pp. 
\bibitem{SJS} S. J. Schreiber, \emph{On growth rates of subadditive functions for semiflows}, J. Differential Equations \textbf{148} (1998), 334--350. 
\bibitem{S} K. V. Storozhuk, \emph{On the Rolewicz theorem for evolution families}, Proc. Amer. Math. Soc. \textbf{135} (2007), 1861--1863. 
\bibitem{SS} R. Sturman and J. Stark, \emph{Semi-uniform ergodic theorems and applications to forced systems}, Nonlinearity \textbf{13} (2000), 113--145.
\bibitem{Walters} P. Walters, \emph{An Introduction to Ergodic Theory}, Springer, Berlin (1981).
\bibitem{Z} J. Zabczyk, \emph{Remarks on the control of discrete-time distributed parameter systems}, SIAM J. Control Optim.  \textbf{12} (1974), 721--735.
\bibitem{ZLZ} L. Zhou, K. Lu and W. Zhang, \emph{Roughness of tempered dichotomies for infinite-dimensional random difference equations}, J. Differential Equations \textbf{254} (2012), 4024--4046.
\end{thebibliography}
\end{document}